\documentclass{amsart}

\usepackage[utf8]{inputenc}
\usepackage[T1]{fontenc}
\usepackage[english]{babel}
\usepackage{amsthm}
\usepackage{amsmath}
\usepackage{mathtools}
\usepackage{tikz-cd}
\usepackage{amsfonts}
\usepackage{amssymb}
\usepackage{graphicx}
\usepackage{xspace}
\usepackage{enumitem}
\usepackage{float}
\floatstyle{plaintop}
\restylefloat{table}

\usepackage[all]{xy}

\usepackage[mathscr]{eucal} % load euler fonts (in amsfonts)
\usepackage{amsmath} %this is needed for \rxar (\xrightarrow)
\usepackage{epsfig}
\usepackage{amscd}
\usepackage{verbatim}
\usepackage{booktabs}

\newtheorem{THEOR}{Theorem}[section]

\newtheorem{LEM}[THEOR]{Lemma}

\newtheorem{COR}[THEOR]{Corollary}
\newtheorem{REM}[THEOR]{Remark}

\theoremstyle{definition}

\numberwithin{equation}{section}

\newcommand\Z{\mathbb Z}

\newcommand\C{\mathbb C}
\newcommand{\PP}{{\mathbb P}}
\def\P{\PP}
\newcommand\A{{\mathsf A}}

\newcommand\Ag{\A_g}
\def\ag{\Ag}

\newcommand{\M}{{\mathsf{M}}}

\newcommand{\rgb}{\mathsf{R}_{g,b}}
\newcommand{\pgb}{\mathsf{P}_{g,b}}

\newcommand{\tG}{{\tilde{G}}}

\newcommand{\tg}{{\tilde{g}}}

\newcommand{\tC} {{\tilde{C}}}

\newcommand{\Hom}{\operatorname{Hom}}

\newcommand{\Spec}{\operatorname{Spec}}
\newcommand{\Nm}{\operatorname{Nm}}

\def\i{\mathrm i}

\renewcommand{\phi}{\varphi}
\newcommand{\sx}{\langle}
\newcommand{\xs}{\rangle}
\newcommand{\lra}{\longrightarrow}
\newcommand{\ra}{\rightarrow}

\newcounter{example}[section]

\begin{document}

	\title{Abelian covers and the second fundamental form}

	 \author[P. Frediani]{Paola Frediani} \address{ Dipartimento di
	Matematica, Universit\`a di Pavia, via Ferrata 5, I-27100 Pavia,
	 Italy } \email{{\tt paola.frediani@unipv.it}}

	\begin{abstract}
	We give some conditions on a family of abelian covers of $\P^1$ of genus $g$ curves, that ensure that the family yields a subvariety of $\ag$ which is not totally geodesic, hence it is not Shimura.  As a consequence, we show that for any abelian group $G$, there exists an integer $M$ which only depends on $G$ such that if $g >M$, then the family  yields a subvariety of $\ag$ which is not totally geodesic. We prove then  analogous results for families of abelian covers of $\tC_t \ra \P^1 = \tC_t/\tG$ with an abelian Galois group $\tG$ of even order, proving that under some conditions, if $\sigma \in \tG$ is an involution,  the family of Pryms  associated with the covers $\tC_t \ra C_t= \tC_t/\langle \sigma \rangle$ yields a subvariety of ${\mathsf A}_{p}^{\delta}$ which is not totally geodesic. As a consequence, we show that if $\tG =(\Z/N\Z)^m$ with $N$ even, and $\sigma$ is an involution in $\tG$, there exists an integer $M(N)$ which only depends on $N$ such that, if   $\tg = g(\tC_t) > M(N)$, then the subvariety of the Prym locus in ${\A}^{\delta}_{p}$ induced by any such family is not totally geodesic (hence it is not Shimura). 

		\end{abstract}

	 \thanks{The author was partially supported MIUR PRIN 2017
	``Moduli spaces and Lie Theory'' ,  by MIUR, Programma Dipartimenti di Eccellenza
	(2018-2022) - Dipartimento di Matematica ``F. Casorati'',
	Universit\`a degli Studi di Pavia and by INdAM (GNSAGA)  }

	% \date{\today}
	\maketitle

	\section{Introduction}
	
	In this paper we study families of abelian covers of $\P^1$, in relation with the Coleman-Oort conjecture and with an analogue of this conjecture for the Prym maps of (possibly ramified) double covers. 
	
	Given a family of Galois covers of genus $g$ curves, $C_t \ra C_t/G \cong \P^1$,  the associated family of Jacobians gives a subvariety of $\ag$ and we ask under which conditions this subvariety is a Shimura subvariety of $\ag$.
A Shimura subvariety of $\ag$ is by definition a Hodge locus for the tautological family of principally polarized abelian varieties on $\ag$.

	 Coleman-Oort conjecture predicts that for $g$ sufficiently high, there should not exist positive dimensional Shimura subvarieties of $\ag$ generically contained in the Torelli locus. In \cite{cfgp} (and in \cite{fg}, \cite{fgm} in the ramified case) a similar question was posed about the existence of positive dimensional Shimura subvarieties of the moduli space of polarised abelian varieties of a fixed dimension, which are generically contained in the Prym loci, that is in the closure of the image of the Prym maps of (possibly ramified) double covers. 
	Both in the case of the Torelli map and in the case of Prym maps, there are examples of such Shimura subvarieties for low values of $g$. These constructions all use families of Galois covers (see \cite{dejong-noot}, \cite{rohde}, \cite{moonen-special}, \cite{moonen-oort},  \cite{fgp}, \cite{fpp}, \cite{gm}, \cite{fgs} for the Torelli case, \cite{cfgp}, \cite{fg}, \cite{fgm} for the Prym case). 
	
Shimura subvarieties are totally geodesic, i.e, they are images of totally
geodesic submanifolds of the Siegel space ${\mathsf H}_g$ endowed with the symmetric metric. Recall that a  submanifold of ${\mathsf H}_g$ is totally geodesic if its second fundamental form is identically zero.  
More precisely, by results of Mumford and Moonen an algebraic subvariety of $\ag$ is a Shimura subvariety if and only if it is totally
geodesic and it contains a CM (complex multiplication) point (see \cite{mumford-Shimura} \cite{moonen-special}). The notion of CM point is arithmetic, while the condition of being totally geodesic is related to the locally symmetric geometry of $\ag$ induced by the Siegel space. 	
	
	Let us first explain the main result of this paper in the case of the Torelli morphism. Fix a family of abelian covers of $\P^1$,  $C_t \ra C_t/G \cong \P^1$, where $C_t$ has genus $g$. We show that  there exists an integer $M$ which only depends on $G$ such that if $g >M$, then the family of Jacobians of the curves $C_t$ yields a subvariety of $\ag$ which is not totally geodesic, hence it is not Shimura. 
	%The moduli space $\ag$ of principally polarised abelian varieties of dimension $g$ is endowed with an orbifold metric which in induced by the symmetric metric on the Siegel space ${\mathsf H}_g$ and a Shimura subvariety of $\ag$ is totally geodesic with respect to this metric (\cite{mumford-Shimura}, \cite{moonen-linearity-1}). 
	
	More precisely, assume we have a family of Galois covers of $\P^1$ with Galois group $G$. With the notations used in \cite[Section 2]{fgp}, denote by $\M_g({\bf{m}}, G,\theta)$ the corresponding subvariety of $\M_g$. Here $\theta: \Gamma_r = \langle \gamma_1, ..., \gamma_r \ | \ \prod_{i=1}^r \gamma_i = 1 \rangle \ra G$ denotes the monodromy of the cover (${\bf m}:= (m_1,...,m_r)$ is such that $m_i$ is the order of $\theta(\gamma_i)$ in $G$). Consider $j: \M_g \ra  \A_g$ the Torelli morphism, which is an orbifold immersion outside the hyperelliptic locus (\cite{oort-st}). The second fundamental form of the Torelli map with respect to the Siegel metric has been studied in \cite{cpt}, \cite{cfg}, \cite{fp}. Its dual at a point $[C] \in \M_g$ corresponding to a non hyperelliptic curve $C$ is a map
	\begin{equation}
	\rho: I_2(K_C) \ra S^2 H^0(C, K_C^{\otimes 2})
	\end{equation}
	where $I_2(K_C)$ is the kernel of the multiplication map $S^2H^0(C, K_C) \ra H^0(C, K_C^{\otimes 2})$, that is, the space of quadrics containing the canonical image of $C$. 

In \cite{moonen-special} in the cyclic case and in \cite[Theorem 3.9]{fgp} for any group $G$, it is proven that if the following condition holds
\begin{gather*}
(*) \ \ \ \ \ dim (S^2H^0(C, K_C))^G = \dim  H^0(C, K_C^{\otimes 2})^G,
\end{gather*}
for a general member $[C] \in \M_g({\bf{m}}, G,\theta)$, then the family of Galois covers yields a Shimura (hence totally geodesic) subvariety of $\A_g$.  Moonen proved in \cite{moonen-special},  using techniques in positive characteristic, that condition $(*)$ is also necessary for a family of cyclic covers of $\P^1$  to give a Shimura subvariety of $\A_g$. This was then generalised in \cite{MZ}, still using positive characteristic techniques, for 1-dimensional families of abelian covers of $\P^1$. Using this sufficient condition $(*)$ and  computer computations, many examples of Shimura subvarieties of $\A_g$ generically contained in the Torelli locus have been constructed in \cite{moonen-special}, \cite{moonen-oort}, \cite{fgp}, \cite{fpp}. Recently in \cite{cgpi} the authors proved that if $g \leq 100$ the examples found in \cite{fgp}, that are all in genus $g \leq 7$, are the only ones satisfying condition $(*)$, thus giving a strong evidence for the Coleman-Oort conjecture (at least for Shimura subvarieties obtained via families of Galois covers of $\P^1$). In \cite{fgs} it is shown that the only families of Galois covers of curves of genus $g' >0$ satisfying condition $(*)$ are the ones found in \cite{fpp} that are all in low genus $g \leq 4$.   
In \cite[Proposition 5.4]{cfg} it is proven that if a family of cyclic covers of $\P^1$ does not satisfy $(*)$ and another condition on the dimension of the eigenspaces for the representation of $G$ on $H^0(C, K_C)$ is satisfied, then the family of cyclic covers gives a subvariety of $\A_g$ which is not  even totally geodesic (hence it is not Shimura).  Here we generalise this result to families of abelian covers of $\P^1$ in Theorem \ref{torelli}. 
Let $C \ra C/G \cong \P^1$ be a general element of a family of abelian covers. Since $G$ is abelian, we have an isomorphism of $G$ with the group of characters  $G^*=\Hom(G,\C^{*})$, hence we identify an element $n$ of $G$ with the corresponding character $\chi \in G^*$ and we denote by $d_n := \dim H^0(C, K_C)_{\chi}$ (see Section 2). 

 We have the following 
\begin{THEOR} (Theorem \ref{torelli})
\begin{enumerate}
\item If there exists an element  $n \in  G \subseteq (\Z/N\Z)^m$ such that $n \neq -n$, $d_n \geq 2$, $d_{-n} \geq 2$, then the subvariety of $\A_g$ induced by the family of abelian covers of genus $g \geq 4$  is not totally geodesic (hence it is not Shimura). 
\item If there exists an element $n  \in  G \subseteq (\Z/N\Z)^m$ of order 2 such that $d_n \geq 3$, then the subvariety of $\A_g$ induced by the family of abelian covers of genus $g \geq 4$ is not totally geodesic (hence it is not Shimura).	 
\end{enumerate}
	
\end{THEOR}

 As a consequence we prove the following 

	\begin{THEOR}
\label{A} (Corollaries \ref{Z2}, \ref{general}, \ref{Zp})

Assume $G \subseteq (\Z/N\Z)^m$, $N \geq 3$, set $d:= \#G$, $g \geq 4$. Assume that we have a family of $G$-covers of $\P^1$ yielding a totally geodesic subvariety of $\A_g$.  Then $r \leq 2Nm$ and $ g \leq 1 + d(m(N-1)-1).$

%Let $G \subseteq (\Z/N\Z)^m$ be any finite abelian group,  set $d:= \#G$. Then there exists an integer $M(d, N)$ which only depends on $d$ and $N$ such that every family of abelian covers of $\P^1$ with Galois group $G$ of genus  $g > M(d,N)$ yields a subvariety of $\ag$ which is not totally geodesic (hence it is not Shimura). 

If $G = (\Z/2\Z)^m$ ($g \geq 4$) and we have a family of $G$-covers of $\P^1$ yielding a totally geodesic subvariety of $\A_g$,  then $m \leq 6$, $r \leq 6m\leq 36$, and $g \leq 1 + 2^{m-1}(3m-2) \leq 513.$

If $G = (\Z/p\Z)^m$, with $p$ a prime number, $p \geq 3$, $g \geq 4$. Assume that we have a family of $G$-covers of $\P^1$ yielding a totally geodesic subvariety of $\A_g$. Then $m \leq 2p$, $r \leq 4 p^2$ and $ g \leq 1 + p^{2p}(2p(p-1) -1).$

\end{THEOR}
	
Note that in \cite[Theorem 5.2]{moh-2} it is proven, using characteristic $p$ techniques, that if there exists $n \in G$ such that $\{d_n, d_{-n}\} \neq \{0, r-2\}$, where $r$ is the number of branch points of the cover,  and $d_n + d_{-n} \geq r-2$, then the subvariety of $\A_g$ induced by the family of abelian covers  is not Shimura. \\

	In the second part of the paper we show similar statements for families of Galois covers yielding subvarieties generically contained in the Prym loci of possibly ramified double covers. 
	
	Denote by $ { \rgb}$ the moduli space of isomorphism classes of triples $[(C, \alpha, B)]$ where $C$ is a smooth complex projective curve of genus $g$, $B$ is a reduced effective divisor of degree $b$ on $C$ and $\alpha \in Pic(C)$ is such that $\alpha^2={\mathcal O}_{C}(B)$. A point  $[(C, \alpha, B)] \in {\rgb}$ determines a double cover of $C$, $f:\tC\to C$ branched on $B$, with $\tC=\Spec({\mathcal O}_{C}\oplus \alpha^{-1})$.
	
	The Prym variety $P(C, \alpha, B)$ (also denoted by $P(\tC,C)$) associated to $[(C,\alpha, B)]$ is the connected component containing the origin of the kernel of the norm map $\Nm_{f}: J\tC \to JC$.
	When $b>0$, $\ker \Nm_{f}$ is connected. The variety  $P(C, \alpha, B)$ is a  polarised abelian variety of dimension $g-1+\frac{b}{2}$. In fact, if we denote by $\Xi$ the restriction to $P(\tC,C)$ of the principal polarisation on $J\tC$,  if $b=0$, $\Xi$ is twice a principal polarisation, hence we consider on $P(\tC, C)$ this principal polarisation. If $b>0$, we endow  $P(\tC, C)$ with the polarisation  $\Xi$ which  is of type $\delta=(1,\dots, 1, \underbrace{2,\dots,2}_{g\text{ times}})$.

Denote by $\A^\delta_{g-1+\frac{b}{2}}$  the moduli space of abelian varieties of dimension $g-1+\frac{b}{2} $ with a polarization of type $\delta$, then we consider the Prym map 
	\[ \pgb: \rgb \lra \A^\delta_{g-1+\frac{b}{2}}, \quad [(C,\alpha,B)] \longmapsto  [(P(C,\alpha,B), \Xi)].\]

	The dual of the differential of the Prym map $\pgb$ at a generic point $[(C, \alpha, B)]$  is given by the multiplication map
	\begin{equation}
	\label{dp}
	(d\pgb)^* : S^2H^0(C, {\omega}_C \otimes \alpha) \to H^0(C, \omega^2_C\otimes\mathcal O_C(B))
	\end{equation}
	
	The multiplication map is surjective if  $\dim \rgb \leq \dim {\A}^{\delta}_{g-1+\frac{b}{2}},$ (\cite{lange-ortega}).  So the Prym map $\pgb$ is generically finite, if and only if $\dim \rgb \leq \dim {\A}^{\delta}_{g-1+\frac{b}{2}},$ that is 
	if either $b\geq 6$ and $g\geq 1$, or $b=4$ and $g\geq 3$, $b=2$ and $g\geq 5$,  $b=0$ and $g \geq 6$.
	
	For $b =0$  the Prym map is generically injective for $g \geq 7$  (\cite{friedman-smith}, \cite{kanev-global-Torelli}). 
	If $b>0$,  in  \cite{pietro-vale},  \cite{mn}, \cite{no}, it is proven that $\pgb$  is generically injective in all the 
	cases except for $b=4$, $g =3$, when the degree is 3 (\cite{nagarama}, \cite{bcv}).  Moreover, a global Prym-Torelli theorem was recently proved for all $g$ and $b\ge 6$ (\cite{ikeda} for $g=1$ and \cite{no1} for all $g$).

	Let $\tG$ be a group containing a central  involution $\sigma$. Assume we have a family of Galois covers $\psi_t: \tC_t \ra \P^1 = \tC_t/\tG$. Then we have an exact sequence
	$ 0 \rightarrow  \langle \sigma \rangle \ra \tG \ra G \ra 0,$ and a commutative diagram  
	
	\begin{equation}
	\label{tc-c}
	\begin{tikzcd}[row sep=tiny]
	\tC_t \arrow{rr}{\phi_ t} \arrow{rd}{\psi_t} & &  \arrow{ld }{\pi_t} C_t   = \tC_t /\sx \sigma\xs  \\
	& \P^1 &
	\end{tikzcd}
	\end{equation}	
	%\label{prymetale}

	For a general element of the family $\tC_t$, denote by $V:= H^0(\tC_t, K_{\tC_t}) \cong V_+ \oplus V_-$, where $V_+$ is the set if $\sigma$-invariants elements, and $V_-$ is the set if $\sigma$-anti-invariants elements. The double cover $\psi_t: \tC_t \ra C_t = \tC_t/\sx \sigma\xs$ corresponds to a triple $(C_t, \alpha_t, B_t)$, where $B_t$ is a reduced effective divisor of degree $b$ on $C_t$ and $\alpha_t$ is a line bundle on $C_t$ such that $\alpha_t^2 = {\mathcal O}_{C_t}(B_t)$. Then $V_+ \cong H^0(C_t, K_{C_t})$ and $V_- \cong H^0(C_t, K_{C_t} \otimes \alpha_t) \cong H^{1,0}(P(\tC_t,C_t))$. 
	
	For the precise construction of the subvarieties of ${\A}^{\delta}_{g-1+\frac{b}{2}}$ generically contained in the Prym loci given by families of Galois covers, see  \cite[ \S 3]{fgm}. 	Given a family of Galois covers as above, observe that multiplication map $$ m : S^2 V \to H^0(\tC, K_{\tC_t}^{\otimes 2} ),$$ which is the dual of the differential of the Torelli map
	$\widetilde{j} : \M_\tg \to \A_\tg$ at the point $[\tC_t] \in \M_\tg$, is $\tG$-equivariant, hence $m$ maps $(S^2V)^\tG $ to $H^0(\tC, K_{\tC_t}^{\otimes 2})^\tG$. We have the following isomorphism:
	$$(S^2V)^\tG = (S^2(V_+))^{\tG} \oplus (S^2 (V_-))^{\tG}. $$
	
	The codifferential of the restriction of  Prym map  to the subvariety of $\rgb$ given by our family of Galois covers at the point $[(C_t, \alpha_t, B_t)]$  is the restriction of the multiplication map $m$ to $(S^2 (V_-))^{\tG}$, that we still denote by $
	m : (S^2 (V_-))^{\tG} \lra H^0(\tC, K_{\tC_t}^{\otimes 2})^\tG$ (see \cite{fg}, \cite{fgm}).

	In \cite[Theorem~3.2]{fg}, which is a generalisation of Theorems 3.2 and 4.2 in \cite{cfgp}, it is shown that if the map $m : (S^2 (V_-))^{\tG} \lra H^0(\tC, K_{\tC_t}^{\otimes 2})^\tG$ is an isomorphism, then the family of Pryms yields a Shimura subvariety of $\A^\delta_{g-1+\frac{b}{2}}$. Using this criterion,  in \cite{cfgp}, \cite{fg}, \cite{fgm} many examples of Shimura subvarieties generically contained in the Prym loci have been constructed. 
	
A natural question is to ask whether the above condition is also necessary for such families of abelian covers to yield Shimura subvarieties of $\A^\delta_{g-1+\frac{b}{2}}$ generically contained in the Prym loci. 

We give a partial answer to this question in the case of abelian covers in Theorem \ref{prym}, which is an analogue of Theorem \ref{torelli} in the Prym case. Notice that Theorem \ref{prym} improves the results obtained in \cite{moh} by different techniques. 

As a consequence we show the following 

\begin{THEOR}
\label{B} (Corollaries \ref{Z2prym}, \ref{general-prym}, and Remark \ref{remark})

Assume $\tG =(\Z/N\Z)^m$ with $N$ even, and let $\sigma$ be an involution in $\tG$. Consider a family of abelian covers of $\P^1$, $\tC_t \ra \tC_t/\tG \cong \P^1$, with $g(\tC_t) = \tg \geq 4$. Assume that the covers $\tC_t \ra C_t = \tC/\langle \sigma \rangle$ have $b$ ramification points and denote by $g$ the genus of  $C_t$. \\
If the multiplication map $m : (S^2 (V_-))^{\tG} \lra H^0(\tC, K_{\tC_t}^{\otimes 2})^\tG$ is surjective at a general element of the family of covers, then there exists an integer $M(N)$ which only depends on $N$ such that, if   $\tg > M(N)$, then the subvariety of the Prym locus in ${\A}^{\delta}_{g-1+\frac{b}{2}}$ induced by any such family is not totally geodesic (hence it is not Shimura).

\end{THEOR}

We remark that the condition on the surjectivity of the  multiplication map 
$$m : (S^2 (V_-))^{\tG} \lra H^0(\tC, K_{\tC_t}^{\otimes 2})^\tG$$  at a general element of the family of covers is automatically satisfied if $b \geq 6$ thanks to the Prym-Torelli theorem proved in \cite{no1}, \cite{ikeda}. 

The technique used to prove the main Theorems \ref{torelli} and \ref{prym} is the computation of the second fundamental forms of the Torelli map and of the Prym maps on some particular quadrics  invariant under the action of the Galois group.  The fact that the Galois group is abelian allows to explicitly construct such quadrics. Then, the techniques developed in \cite{cpt}, \cite{cfg} for the second fundamental form of the Torelli map and in \cite{cf1}, \cite{cf2} for the second fundamental form of the Prym map, allow to compute these second fundamental forms on the quadrics that we have constructed and show that they do not vanish. A similar technique in the case of cyclic groups, has been used in \cite{fp} to show that the bielliptic and the bihyperelliptic loci are not totally geodesic.

The structure of the paper is as follows. In section 2 we recall the construction of abelian covers of $\P^1$ following \cite{MZ}, \cite{W}.  In section 3 we prove the main results in the case of the Torelli map and we finish giving one example. In section 4 we prove the main results in the case of the Prym maps and we conclude giving some examples.

		\section*{Acknowledgements} I would like to thank Alessandro Ghigi for useful discussions and  Roberto Pignatelli for suggestions that improved the estimates given in Corollaries \ref{Z2}, \ref{Zp}.

	\section{Abelian covers of $\P^1$}
	\label{Abelian covers of line}
	
	In this section, we recall the construction of abelian covers of $\P^1$ following \cite{MZ}, \cite{W} (see also \cite{P2}, \cite{fgm}).

	Consider an $m\times r$ matrix $A=(r_{ij})$ whose entries $r_{ij}$ are in $\Z/N\Z$ for some $N\geq 2$ such that the sum of the columns of $A$ is zero  in $(\Z/N\Z)^m$ and all the columns are different from zero. 
Denote by $\widetilde{r}_{ij}$ the lift of $r_{ij}$ to $\Z \cap [0,N)$ and take $t_1,...,t_r\in \C$, $t_i \neq t_j$, $\forall i \neq j$. 
	Let $\overline{\C(x)}$ be the algebraic closure of $\C(x)$. For each $i=1,\dots,m,$ choose a function $w_{i}\in\overline{\C(x)}$ with
	\begin{equation}\label{equation abelian}
	w_{i}^{N}=\prod_{j=1}^{r}(x-t_{j})^{\widetilde{r}_{ij}}\quad\text{for }i=1,\dots, m,
	\end{equation}
	in $\C(x)[w_{1},\dots,w_{m}]$. 
		The equations  \eqref{equation abelian} give in general a singular affine curve and we take its normalization.  Notice that the cover given by \eqref{equation abelian} is not ramified over the infinity and the local monodromy around the branch point $t_{j}$ is given by the column vector $(r_{1j},\dots, r_{mj})^{t}$. 
	Hence the order of ramification over $t_{j}$ is $\frac{N}{\gcd(N,\widetilde{r}_{1j},\dots,\widetilde{r}_{mj})}$. The covering has the Galois group which is given by the subgroup $G$ of $(\Z/N\Z)^m$ generated by the columns of $A$. 
	So by the Riemann-Hurwitz formula, we compute the genus of the curve as follows: 
		\begin{equation}
	g=1+d\left(\frac{r-2}{2}-\frac{1}{2N}\sum_{j=1}^{r}\gcd(N,\widetilde{r}_{1j},\dots,\widetilde{r}_{mj})\right),
	\end{equation}
	where $d$ is the cardinality of $G$.  Denote by $G^*=\Hom(G,\C^{*})$ the group of the characters of $G$. 
	Consider a Galois covering $\pi: C\rightarrow \P^{1}$ with Galois group $G$. 
	The group $G$ acts on the sheaves $\pi_{*}(\mathcal{O}_C)$ and $\pi_{*}(\C)$ via its characters and we get corresponding eigenspace decompositions 
	$\pi_{*}(\mathcal{O}_C)=\bigoplus_{\chi \in G^*} \pi_{*}(\mathcal{O}_C)_{\chi}$ and $\pi_{*}(\C)=\bigoplus_{\chi \in G^*}\pi_{*}(\C)_{\chi}$. 
	Set $L^{-1}_{\chi}=\pi_{*}(\mathcal{O}_{C})_{\chi}$ and let $\C_{\chi}= \pi_{*}(\C)_{\chi}$ denote the eigensheaves corresponding to the character $\chi$. Then 
	$L_{\chi}$ is a line bundle and $\C_{\chi}$ is a local system of rank 1 outside of the branch locus of $\pi$. 
	
	%\begin{remark} \label{abeliangroupcharacter}
	%	Let $G$ be a finite abelian group, then the character group $G^*=\Hom(G,\C^{*})$ is isomorphic to $G$. 
	%	To see this, first assume that $G=\Z/N\Z$ is a cyclic group. 
	%	Fix an isomorphism between $\Z/N\Z$ and the group of $N$-th roots of unity in $\C^{*}$ via $1\mapsto \exp(2\pi \i/N)$. 
	%	Now the group $G^*$ is isomorphic to this latter group via $\chi\mapsto \chi(1)$. 
	%	In the general case, $G$ is a product of finite cyclic groups, so this isomorphism extends to an isomorphism $\varphi_G: G \xrightarrow{\sim} G^*$. 
	%	In the sequel, we use this isomorphism frequently to identify elements of $G$ with its characters.
	%\end{remark}

	Denote by $l_{j}$ be the $j$-th column of the matrix $A$. The group $G$ is the subgroup of $(\Z/N\Z)^m$ generated by the columns of $A$, hence  each column $l_{j}$ is an element of $G$. Since $G$ is a finite abelian group, then the character group $G^*=\Hom(G,\C^{*})$ is isomorphic to $G$. In fact, if $G=\Z/N\Z$ is cyclic, then  
	we can fix an isomorphism between $\Z/N\Z$ and the group of $N$-th roots of unity in $\C^{*}$ via $1\mapsto e^{\frac{2\pi \i}{N}}$.  Therefore $G^*$ is isomorphic to the group of $N$-th roots of unity via $\chi\mapsto \chi(1)$.  In general, the abelian group $G$ is a product of finite cyclic groups, so this isomorphism extends to an isomorphism $G \xrightarrow{\sim} G^*$. 
		Fix a character $\chi$ of $G$, then $\chi(l_{j})\in \C^{*}$ and since $G$ is finite, $\chi(l_{j})$ is a root of unity. Hence there exists a unique integer  $\alpha_{j} \in [0,N)$ such that $\chi(l_{j})=e^{\frac{2\pi\i\alpha_{j}}{N}}$. 
	Equivalently, we can obtain $\alpha_j$ as follows: let $n\in G\subseteq (\Z/N\Z)^{m}$ be the element corresponding to $\chi$ under the above isomorphism. 
	We see $n$ a row vector. Then  $(\alpha_{1},\dots,\alpha_{r})= n\cdotp A$. 
	
	 Denote by $\widetilde{n}$ the lift of $n$ to $(\Z\cap [0,N))^{m}$ and set $\widetilde{n}\cdotp \widetilde{A}=(\overline{\alpha}_{1},\dots,\overline{\alpha}_{r})$, where $\tilde{A}$ is the matrix with entries given by the $\tilde{r}_{ij}$'s. This means that $\overline\alpha_j = \sum_{i=1}^m n_i \tilde{r}_{ij} \in \Z$ (notice that $\overline\alpha_j$ is not necessarily in $\Z\cap [0,N)$).
		
	Let us denote by $K_C$ the canonical sheaf of $C$. 
	The sheaf $\pi_{*}(K_C)_{\chi}$ decomposes according to the action of $G$. Let $\chi$ be the character associated to an element $n \in G$. Then we have: 
		$L_{\chi}=\mathcal{O}_{\P^{1}}(\sum_{j=1}^{r}\langle\frac{\alpha_{j}}{N}\rangle)$, 
		where $\langle x\rangle$ is the fractional part of the real number $x$ and
		\begin{equation}
		\label{formula}
		\pi_{*}(K_C)_{\chi}= K_{\P^1} \otimes L_{\chi}^{-1}=\mathcal{O}_{\P^1}\left(-2+\sum_{j=1}^{r}\left\langle -\frac{\alpha_{j}}{N}\right\rangle\right).
		\end{equation}
			For a proof of this, see  \cite[Lemma 4.2]{fgm}, \cite[Proposition 1.2]{P2}.  
Consider now an abelian group $G\subseteq (\Z/N\Z)^{m}$  and a $G$-abelian cover given by the equations \eqref{equation abelian}. 
	Let $n\in G$ be the element $(n_1,\dots, n_m)\in (\Z/N\Z)^{m}$ under the inclusion $G\subseteq (\Z/N\Z)^{m}$ with $n_i\in\Z\cap[0,N)$. 
	By \eqref{formula}, we have 
	\begin{equation}
	\label{dn}
	d_n:= \dim H^0(C,K_{C})_{n}=-1+\sum_{j=1}^{r}\langle-\frac{\alpha_{j}}{N}\rangle.
	\end{equation}
	A basis for  $H^0(C, K_{C})$ is given by the forms 
	\begin{equation}
	\label{forms} 
	\omega_{n,\nu}=x^{\nu} w_{1}^{n_1}\cdots w_{m}^{n_m}\prod_{j=1}^{r} (x-t_j)^{\lfloor -\frac{\overline\alpha_j}{N}\rfloor}dx.
	\end{equation}
	Here $\overline\alpha_j$ is as above and $0\leq\nu\leq d_{n}-1=-2+\sum_{j=1}^{r}\langle-\frac{\alpha_{j}}{N}\rangle$. 
	The fact that these elements give a basis of $H^0(C, K_{C})$ is proven in \cite[Lemma 5.1]{MZ}.

To construct a family of abelian  covers of $\P^1$, we fix a matrix $A$ as above and we let the points $(t_1,...,t_r)$ vary in $Y_r:= \{(t_{1},\dots,t_{r})\in(\mathbb{A}_{\C}^{1})^{r}\mid t_{i}\neq t_{j}, \  \forall i\neq j \}$. 
	Over this affine open set we define a family of abelian covers of $\P^1$ by the equation \eqref{equation abelian}. Clearly the 
	 branch points are $(t_{1},\dots,t_{r})\in Y_r$. Varying the branch points we get a family $f:\tilde{\mathcal C}\to Y_r$ of smooth projective curves whose fibers $C_t$ are the abelian covers of $\P^1$ introduced above.  Denote by $ \M_{g}(G, A)  \subset \M_g$ the corresponding subvariety in $\M_g$. For more details on the construction of families of Galois covers see e.g. \cite{fgp}, \cite{baffo-linceo}, \cite{clp2},  \cite{fgm}.

	\section{Totally geodesic subvarieties in the Torelli locus}
	\label{torelli}
	
	Assume we have a family of Galois covers of $\P^1$ with Galois group $G$. With the notation used in \cite[Section 2]{fgp}, denote by $\M_g({\bf{m}}, G,\theta)$ the corresponding subvariety of $\M_g$. Here $\theta: \Gamma_r = \langle \gamma_1, ..., \gamma_r \ | \ \prod_{i=1}^r \gamma_i = 1 \rangle \ra G$ denotes the monodromy of the cover (${\bf m}:= (m_1,...,m_r)$ is such that $m_i$ is the order of $\theta(\gamma_i)$ in $G$). Consider $j: \M_g \ra  \A_g$ the Torelli morphism, which is an orbifold immersion outside the hyperelliptic locus. We endow $\A_g$ with the Siegel metric, that is the orbifold metric induced on $\A_g$ by the symmetric metric on the Siegel space ${\mathsf H_g}$. 

Recall that an algebraic subvariety of $\A_g$ is  totally geodesic if it is the image of  a totally geodesic submanifold of the Siegel space ${\mathsf H}_g$ endowed with the symmetric metric. A  submanifold of ${\mathsf H}_g$ is totally geodesic if its second fundamental form is identically zero.  
Moreover,  an algebraic subvariety of $\ag$ is a Shimura subvariety if and only if it is totally
geodesic and it contains a CM (complex multiplication) point (see \cite{mumford-Shimura} \cite{moonen-special}). Hence, if an algebraic subvariety of $\A_g$ is not totally geodesic, it is not Shimura. 	
	
	The second fundamental form of the Torelli map with respect to the Siegel metric has been studied in \cite{cpt}, \cite{cfg}, \cite{fp}. Its dual at a point $[C] \in \M_g$ corresponding to a non hyperelliptic curve $C$ is a map
	\begin{equation}
	\rho: I_2(K_C) \ra S^2 H^0(C, K_C^{\otimes 2})
	\end{equation}
	where $I_2(K_C)$ is the kernel of the multiplication map $S^2H^0(C, K_C) \ra H^0(C, K_C^2)$.
In \cite{moonen-special} in the cyclic case and in \cite[Theorem 3.9]{fgp} for any group $G$, it is proven that if, for a general curve $C$ of our family of Galois covers,  the following condition holds
\begin{gather*}
(*) \ \ \ \ \ \dim (S^2H^0(C, K_C))^G = \dim  H^0(C, K_C^{\otimes 2})^G,
\end{gather*}
 then the family of Galois covers yields a Shimura (hence totally geodesic) subvariety of $\A_g$.  We want to show if $(*)$ does not hold, the group $G$ is abelian, and under some assumptions on the dimension of eigenspaces of the action of $G$ on $H^0(K_C)$, the family of abelian covers is not totally geodesic.
 We will do it by an explicit computation of the second fundamental form of the subvariety along some tangent directions and showing that it does not vanish. 
 
 To do this we will need to compute the first  and the second Gaussian maps of the canonical line bundle $K_C$.

 We briefly recall their definition in local coordinates. Take a local coordinate $z$ on $C$. 
 The first Gaussian (or Wahl map) is the linear map
 $$\mu_1: \wedge^2 H^0(C, K_C) \ra H^0(C, K_C^{\otimes 3}),$$
 $$\mu_1(\alpha \wedge \beta) := (f'(z) g(z) - f(z) g'(z) ) (dz)^3,$$
 where $\alpha = f(z) dz$, $\beta = g(z) dz$ are the local expressions of the holomorphic forms $\alpha$ and $\beta$. Then one immediately sees that the zero divisor of the section $\mu_1(\alpha \wedge \beta)$ is given by $2F +R$, where $F$ is the base locus of the pencil $\langle \alpha, \beta \rangle$, and $R$ is the ramification divisor of the map given by the pencil $\langle \alpha, \beta \rangle$.

Let us now recall the definition of  the second Gaussian map 
$$\mu_2: I_2(K_C) \ra H^0(C, K_C^{\otimes 4}). $$
Take a basis $\{\omega_1,..., \omega_g\}$ of $H^0(C, K_C)$ and a quadric $Q = \sum_{i,j=1}^g a_{ij} \omega_i \otimes \omega_j \in  I_2(K_C)$, where $a_{ij} = a_{ji}$, $\forall i,j$. Assume that locally we have $\omega_i = f_i(z) dz$. Since $Q \in  I_2(K_C)$, we have 
$$ \sum_{i,j=1}^g a_{ij} f_i(z) f_j(z) =0.$$ Derivating we get 
$$ \sum_{i,j=1}^g a_{ij} f'_i(z) f_j(z) =0,$$ 
hence 
$$ \sum_{i,j=1}^g a_{ij} f'_i(z) f'_j(z) +  \sum_{i,j=1}^g a_{ij} f''_i(z) f_j(z) =0 .$$ 
 In local coordinates the second Gaussian map is defined as follows: 
 $$\mu_2(Q) :=  \sum_{i,j=1}^g a_{ij} f'_i(z) f'_j(z) (dz)^4 = - \sum_{i,j=1}^g a_{ij} f''_i(z) f_j(z) (dz)^4.$$

Assume now that the Galois group is an abelian group $G \subseteq  (\Z/N\Z)^m$ of cardinality $d$, and the family of abelian covers has equations given by  \eqref{equation abelian} given by an $m\times r$ matrix $A=(r_{ij})$ whose entries are in $\Z/N\Z$ for some $N\geq 2$. For an element $n \in G$ denote by $V_n := H^0(C, K_C)_n$ and by $d_n = \dim V_n = -1+\sum_{j=1}^{r}\langle-\frac{\alpha_{j}}{N}\rangle$ (see \eqref{dn}). We have the following 
\begin{THEOR}
\label{torelli}
\begin{enumerate}
\item If there exists an element $n \in  G \subseteq (\Z/N\Z)^m$ such that $n \neq -n$, $d_n \geq 2$, $d_{-n} \geq 2$, then the subvariety of $\A_g$ induced by the family of abelian covers of genus $g \geq 4$  is not totally geodesic (hence it is not Shimura). 
\item If there exists an element $n  \in  G \subseteq (\Z/N\Z)^m$ of order 2 such that $d_n \geq 3$, then the subvariety of $\A_g$ induced by the family of abelian covers of genus $g \geq 4$ is not totally geodesic (hence it is not Shimura).	 
\end{enumerate}
	
\end{THEOR}

\begin{proof}

Assume first that the generic curve $C$ of the family of abelian covers is not hyperelliptic. 
In case $(1)$, by assumption, and by \eqref{forms} there exists subspaces $\langle \omega_1, \omega_2 = x \omega_1 \rangle \subseteq V_n$, $\langle \omega_3, \omega_4 = x \omega_3 \rangle \subseteq V_{-n}$. Then the quadric $$Q:= \omega_1 \odot \omega_4  -\omega_2 \odot \omega_3$$ clearly lies in $(I_2(K_C))^G$. Consider the $d:1$ cover $ \pi: C \ra C/G \cong \P^1$ and take a generic fibre $\pi^{-1}(t) = \{p_1,...,p_d\}$ such that $p_i \neq p_j$, $\forall i\neq j$ and such that $\forall i$, $p_i$ does not belong to the base locus of the pencil given by $\langle \omega_1, \omega_3\rangle$, nor to the ramification locus of the map given by the pencil $\langle \omega_1, \omega_3\rangle$. 

Consider the vector $v = \sum_{i=1}^d \xi_{p_i} \in H^1(T_C)^G$, where $\xi_{p_i}$ denotes the Schiffer variation at $p_i$ (for a definition of it see e.g. \cite[section 2.2]{cfg}) . Then using \cite[Theorem 3.1]{cpt}, or \cite[Theorem 2.2]{cfg} and the fact that $\rho$ is $G$-equivariant,  we have 
\begin{equation}
\rho(Q)(v \odot v) = \sum_{i \neq j} \rho(Q)(\xi_{p_i}\odot \xi_{p_j}) + \sum_{i=1}^d  \rho(Q)(\xi_{p_i}\odot \xi_{p_i}) = -4 \pi \i \sum_{i \neq j} Q(p_i, p_j) \cdot \eta_{p_i}(p_j) + d  \rho(Q)  (\xi_{p_1}\odot \xi_{p_1}).
\end{equation}
For $i \neq j$, we have $Q(p_i,p_j) =0$, since $x(p_i) = x(p_j) =t$, while $\rho(Q)  (\xi_{p_1}\odot \xi_{p_1}) = -2 \pi \i \cdot \mu_2(Q)(p_1)$, where $\mu_2: I_2(K_C) \ra H^0(C, K_C^{\otimes 4}) $ is the second Gaussian map of the canonical line bundle $K_C$. 

So we have 
\begin{equation}
\rho(Q)(v \odot v) = -2d \pi i \mu_2(Q)(p_1).
\end{equation}
Assume that in a local coordinate  around $p_1$, $\omega_i = f_i(z) dz$, with $f_2(z) = x(z) f_1(z)$, $f_4(z) = x(z) f_3(z)$.  Then we have

\begin{align*}
\mu_2(Q) =& (f'_1(z) \cdot f'_4(z) - f'_2(z) \cdot f'_3(z))(dz)^4 = \\
=&\big{(}f'_1\cdot (x'(z)f_3(z) + x(z)f'_3(z)) - (x'(z)f_1(z) + x(z) f'_1(z)) \cdot  f'_3(z)\big{)}(dz)^4 = \\   =&x'(z) \cdot \big{(}f'_1(z) f_3(z) -f_1(z) f'_3(z) \big{)}(dz)^4 = \\= &x'(z)dz \cdot \mu_1(\omega_1 \wedge \omega_3),
\end{align*}

where $\mu_1: \wedge^2H^0(C, K_C) \ra H^0(C, K_C^{\otimes 3}) $ denotes the first Gaussian map of the canonical bundle $K_C$. 
So $\mu_2(Q)$ vanishes exactly in the base locus of the linear system $\langle \omega_1, \omega_3\rangle$, in the ramification points of the map given by $\langle \omega_1, \omega_3\rangle$ and in the ramification points of the cover $\pi: C \ra \P^1$. Hence by our choice of $t$, we get $\mu_2(Q)(p_1) \neq 0$, so $\rho(Q)(v \odot v) \neq 0$, so the variety given by the family of abelian covers is not totally geodesic.

In fact, call $X$ the subvariety of $\ag$ given by the family of abelian covers. Then 
the second fundamental form of $X$ in $\ag$ is a map $$S^2 T_X \ra N_{X/\ag},$$
where $T_X$ is the tangent bundle of $X$ and $ N_{X/\ag}$ is the normal bundle of $X$ in $\ag$. Its dual is a map $$\rho_X: N^*_{X/\ag} \ra S^2 T^*_X.$$ Clearly the conormal bundle $N^*_{{{\M_g}/\ag}|X}$ of $\M_g$ in $\ag$ restricted to $X$  is contained in the conormal bundle $N^*_{X/\ag}$ of $X$ in $\ag$. So, at a point $[C] \in X$, we have 
$$I_2(K_C) = N^*_{{{\M_g}/\ag}, [C]} \subset N^*_{{X/\ag}, [C]}.$$
Since $v \in H^1(C, T_C)^G = T_{X, [C]}$  is  tangent to  $X$, we have 
$$\rho_X(Q)(v \odot v) = \rho(Q)(v \odot v) \neq 0.$$ 
So we have shown that the dual $\rho_X$ of the second fundamental form of the subvariety $X$ is not identically zero, hence $X$ is not totally geodesic. 

In case (2) by assumption, there exists a  subspace $\langle \omega_1, \omega_2 = x \omega_1, \omega_3 = x^2 \omega_1 \rangle \subseteq V_n$, with $n = -n$. Thus we can take the quadric $$Q:= \omega_1 \odot \omega_3 - \omega_2 \odot \omega_2 \in I_2(K_C)^G.$$ For a general fibre $\pi^{-1}(t) =\{p_1,...,p_d\}$, take $v = \sum_{i=1}^d \xi_{p_i} \in H^1(T_C)^G$ as above,  then  
\begin{equation}
\rho(Q)(v \odot v)  = -4 \pi \i \sum_{i \neq j} Q(p_i, p_j) \cdot \eta_{p_i}(p_j) + d \rho(Q)  (\xi_{p_1}\odot \xi_{p_1}) = -2d \pi i \mu_2(Q)(p_1).
\end{equation}
If we write in local coordinates as above $\omega_i= f_i(z) dz$, $i=1,2,3$, we have $f_2(z) = x(z) f_1(z)$, $f_3(z) = (x(z))^2 f_1(z)$. So we compute 
\begin{align*} 
\mu_2(Q) =& f'_1(z) \cdot \big{(}2 x(z) x'(z) f_1(z) + (x(z))^2f'_1(z)\big{)}\\
 - &\big{(}x'(z) f_1(z) + x(z) f'_1(z) \big{)}\cdot \big{(}x'(z) f_1(z) + x(z) f_1(z)\big{)}\\
=& - (x'(z))^2( f_1(z))^2.
\end{align*} 

Therefore if we take $t$ generic  as above, we have $\mu_2(Q)(p_1) \neq 0$, hence $\rho(Q)(v \odot v) \neq 0$.

It remains to consider the case in which the family of abelian covers is contained in the hyperelliptic locus. The Torelli map restricted to the hyperelliptic locus ${\mathsf {HE}}_g$, $j_{h}: {\mathsf {HE}}_g \rightarrow {\mathsf A}_g$,  is an orbifold immersion (\cite{oort-st}) and in \cite[Prop. 5.1]{cf-trans}, \cite[Section 6]{fp} the second fundamental form of the restriction of the Torelli map to the hyperelliptic locus has been studied.

We have the following tangent bundle exact sequence

\begin{gather}
\label{tangent}
  \xymatrix{& & 0 \ar[d] & & &\\
    &0\ar[r] &  T_{{\mathsf {HE}}_g}\ar[d]\ar[r] &  {T_{{\mathsf A}_g}}_{|{\mathsf{HE}_g}}\ar[d]^=\ar[r] & N_{{\mathsf {HE}}_g/{\mathsf A}_g} \ar[r]\ar[d] & 0\\
    &  &{T_{{\mathsf M}_g}}_{|{\mathsf HE}_g} \ar[r]&  {T_{{\mathsf A}_g}}_{|{\mathsf{HE}_g}}\ar[r]& {N_{{\mathsf {M}}_g/{\mathsf A}_g}}_{| {{\mathsf {HE}}_g}}
     \ar[r] &0\\
    & & & &  &}
\end{gather}

Denote by 
\begin{equation}
\rho_{HE}: N^*_{{\mathsf {HE}}_g|{\mathsf A}_g} \rightarrow S^2 T^*_{{\mathsf {HE}}_g}
\end{equation}
the dual of the second fundamental form of $j_h$. 

At a point $[C] \in {\mathsf {HE}}_g$, the dual of \eqref{tangent} is

\begin{equation}
\label{cotangent}
  \xymatrix{
        &0 \ar[r] & I_2(K_{C}) \ar[d]^=\ar[r]& S^2 H^0(K_{C})\ar[d]^=\ar[r]^m& H^0(K_{C}^{\otimes 2})\ar[d] &\\
    &0\ar[r] & I_2(K_{C})\ar[r] &  S^2 H^0(K_{C})\ar[r]^m &  H^0(K_{C}^{\otimes 2})^+\ar[r] & 0}
\end{equation}

where $H^0(C, K_C^{\otimes 2})^+$ denotes the subspace of  $H^0(C, K_C^{\otimes 2})$ of the elements which are invariant under the hyperelliptic involution $\tau$ and $I_2(K_C)$ can be identified with the set of quadrics containing the rational normal curve.  

Denote by $H^1(T_C)^+$ the subspace of $H^1(T_C)$ of the elements which are invariant under the hyperelliptic involution $\tau$, that is the tangent space of the hyperelliptic locus at the point $[C]$. We have $\forall Q \in I_2(K_C)$, $\forall v,w \in H^1(T_C)^+$, 
$$\rho_{HE}(Q) (v \odot w) = \rho(Q)(v \odot w)$$ (see \cite[Prop. 5.1]{ cf-trans}).  
%Here  we denote by $\rho(Q)$ the section $ Q \cdot \hat{\eta} $, seen as an element in $S^2 H^0(K_C^{\otimes 2})$ as in \cite[Theorem 3.7]{cfg}. 

So, if the the hyperelliptic involution is contained in the Galois group $G$ of the cover, the tangent vector $v = \sum_{i=1}^d \xi_{p_i}$ is $G$-invariant, hence $\tau$-invariant, i.e. $v \in H^1(T_C)^+$. Therefore if we take a quadric $Q$ as above (both in case (1), and (2)), we have: 

$$
\rho_{HE}(Q)(v \odot v) = \rho(Q)(v \odot v) = \sum_{i \neq j} \rho(Q)(\xi_{p_i}\odot \xi_{p_j}) + \sum_{i=1}^d  \rho(Q)(\xi_{p_i}\odot \xi_{p_i}) = $$
$$=-4 \pi \i \sum_{i \neq j} Q(p_i, p_j) \cdot \eta_{p_i}(p_j) + d  \rho(Q)  (\xi_{p_1}\odot \xi_{p_1}) = -2 \pi \i d \mu_2(Q)(p_1) \neq 0
$$
for a generic choice of the point $t \in \P^1$. 

If the hyperelliptic involution $\tau$ is not contained in $G$, we can consider the subgroup $\tG$ of $Aut(C)$ generated by $G$ and $\tau$. Since $\tau$ is central in $Aut(C)$, the group $\tG$ is abelian and, substituting $G$ with $\tG$ we can assume that the hyperelliptic involution is contained in the Galois group of the cover. By this we mean that we take the quadrics $Q$ as above, we consider the map $h: C \ra C/\tG \cong \P^1$, we take a generic point $t \in \P^1$ and the vector $v = \sum_{i=1}^{2d} \xi_{q_i} \in H^1(T_C)^{\tG} \subset H^1(TC)^+$, where $h^{-1}(t) = \{q_1,..., q_{2d}\}$. Then we have
$$
\rho_{HE}(Q)(v \odot v) = \rho(Q)(v \odot v) = \sum_{i \neq j} \rho(Q)(\xi_{q_i}\odot \xi_{q_j}) + \sum_{i=1}^d  \rho(Q)(\xi_{q_i}\odot \xi_{q_i}) = $$
$$=-4 \pi \i \sum_{i \neq j} Q(q_i, q_j) \cdot \eta_{q_i}(q_j) + d  \rho(Q)  (\xi_{q_1}\odot \xi_{q_1}) = -2 \pi \i d \mu_2(Q)(q_1) \neq 0,
$$
for $t$ general. 
This concludes the proof. \end{proof}

\begin{COR}
\label{Z2}
Under the above assumptions, if $G = (\Z/2\Z)^m$ ($g \geq 4$) and the family is totally geodesic, then 
\begin{enumerate}
\item $r \leq 6m$, and $g \leq 1 + 2^{m-1}(3m-2),$  
\item $m \leq 6$, hence $r \leq 36$, $g \leq 513$.  

\end{enumerate}
\end{COR}
\begin{proof}
 By Theorem \ref{torelli} (2),  if the family is totally geodesic,  for every $n \in G$, with $n \neq 0$, we must have $d_n \leq 2$. 
Consider the i-th row $(r_{i1},...,r_{i r})$ of the matrix $A$ that gives the monodromy of the cover $\pi: C \ra C/G \cong \P^1$.  Denote by  $\tilde{r}_{ij}$ the lift of $r_{ij}$ in $\Z \cap [0,N)$. Set $e_i = (0,...,0,1,0,...,0)^t \in G$ and denote by $\beta_i$ the number of nonzero entries in the row $(r_{i1},...,r_{i r}) = e_i \cdot A$.  Then 
$$2 \geq d_{e_i} = -1 + \sum_{j=1}^r\big{ \langle} \frac{-\tilde{r}_{ij}}{2} \big{\rangle} =  -1 +\frac{\beta_i}{2} ,$$
hence $\beta_i \leq 6$, $\forall i =1,...,m$. So if $r >6m$ there must be a column of $A$ which is zero, a contradiction. Recall that the columns of $A$ give the monodromy of the cover, hence they have order 2.  So $r \leq 6m$ and by the Riemann Hurwitz formula we have: 
$$2g-2 = -2 \cdot 2^m + r\cdot \frac{2^m}{2}, $$
so $$g = 1 + 2^{m-1}(\frac{r}{2} -2) \leq 1 + 2^{m-1}(3m-2).$$

It remains to show that $m \leq 6$. Since $G=(\Z/2\Z)^m$ is generated by the columns of the matrix $A$, the set of the columns of $A$ contains a basis of the vector space $ (\Z/2\Z)^m$. Hence, composing with an automorphism of $(\Z/2\Z)^m$, we can assume that the canonical basis $e_1,...,e_m$ is a subset of the set of the columns of $A$. So, consider the element $n:= (1,...,1)^t$, seen as a character $\chi: G \ra \{ \pm1\} \subset \C^{*}$.  Then $(1,...,1) A = (\alpha_1,...,\alpha_r) $ has at least $m$ non zero entries, corresponding to the columns given by $e_1,...,e_m$. So one immediately computes $d_n \geq -1 + \frac{m}{2}$, and by Theorem \ref{torelli} (2), we must have $m \leq 6$. Thus $r \leq 6m \leq 36$ and $g \leq 1 + 2^{m-1}(3m-2) \leq 513$. 
%Let $H := Ker (\chi)$ and set $C':= C/H$. Then the double cover $\phi: C' \ra C'/(G/H)$ ramifies at least over the $m$ critical values of the cover $C \ra C/G$ whose local monodromy is given by the columns of $A$ corresponding to the basis $e_1,...,e_m$. Then applying Riemann Hurwitz formula to the cover $\phi$, we get: $2g(C') -2 \geq  -4 + m$, so $g(C') \geq \frac{m}{2} -1$. Since $g(C') = \dim H^0(C, K_{C})_{n} = d_{n}$, we get $d_{n} =g(C') \geq \frac{m}{2} -1$. 

\end{proof}

\begin{COR}
\label{general}
Assume $G \subseteq (\Z/N\Z)^m$, $N \geq 3$, set $d:= \#G$, $g \geq 4$. Assume that we have a family of $G$-covers of $\P^1$ yielding a totally geodesic subvariety of $\A_g$.  Then $r \leq 2Nm$ and $ g \leq 1 + d(m(N-1)-1).$
 \end{COR}	

\begin{proof} 
Consider as above the element $e_i=(0,...,0, 1,0,...,0)^t$, then $e_i^t \cdot A = (r_{i1},...,r_{ir})$ is the  $i^{th}$-row of $A$. By Theorem \ref{torelli}, for each row $e_i^t \cdot A$ of the matrix $A$ we must have $d_{e_i} \leq 2$, if $e_i^t \cdot A$ has order 2, otherwise either $d_{e_i} \leq 1$, or $d_{-e_i} \leq 1$. 

We claim that in the first case we obtain $\beta_i \leq 6$, while in the second case $\beta_i \leq 2N$. 

In fact if $e_i^t \cdot A$ has order 2, then all its nonzero entries are equal to $N/2$, hence 
$2 \geq d_{e_i} = -1 + \sum_{j=1}^r\big{ \langle} \frac{-\tilde{r}_{ij}}{N} \big{\rangle} =  -1 +\frac{\beta_i}{2} ,$ so $\beta_i \leq 6$. 

If $e_i^t \cdot A$ has order bigger than 2, we have either 
$$1 \geq  d_{e_i} = -1 + \sum_{j=1}^r \big{\langle} \frac{-\tilde{r}_{ij}}{N} \big{\rangle} = -1 + \sum_{j |  \tilde{r}_{ij} \neq 0}(1 -  \frac{\tilde{r}_{ij}}{N} ) \geq  -1+\beta_i- \beta_i \frac{N-1}{N},  $$
(since $\tilde{r}_{ij} \leq N-1$), $\forall i,j$, thus $\beta_i\leq 2N, $

or 
$$1 \geq  d_{-e_i} = -1 + \sum_{j=1}^r \big{\langle} \frac{\tilde{r}_{ij}}{N} \big{\rangle} = -1 + \sum_{j |  \tilde{r}_{ij} \neq 0}\frac{\tilde{r}_{ij}}{N}  \geq  -1+\frac{\beta_i}{N}, $$
thus again we have $\beta_i \leq 2N. $

Denote by $p$ the number of rows of $A$ of order 2, and by $q$ the number of rows of $A$ of order greater than 2.  Then we must have $r \leq \sum_{i=1}^m \beta_i \leq 6p+2Nq \leq 2N(p+q) = 2Nm$.

By the Riemann Hurwitz formula we have: 
$$2g-2 = d\big{(}-2 + \sum_{i=1}^r (1 - \frac{1}{m_i})\big{)} \leq -2d + dr(1 - \frac{1}{N}) \leq -2d +2dNm(1 - \frac{1}{N}) = 2d(m(N-1)-1),$$
since $m_i \leq N$, $\forall i$ and $r \leq 2Nm$. Hence $g \leq 1 + d(m(N-1)-1).$

\end{proof}

\begin{COR}
\label{Zp}
Assume $G = (\Z/p\Z)^m$, with $p$ a prime number, $p \geq 3$, $g \geq 4$. Assume that we have a family of $G$-covers of $\P^1$ yielding a totally geodesic subvariety of $\A_g$.  Then $m \leq 2p$, $r \leq 4 p^2$ and $ g \leq 1 + p^{2p}(2p(p-1) -1).$
\end{COR}

\begin{proof}
By Corollary \ref{general} we know that $r \leq 2pm$ and $ g \leq 1 + p^m(m(p-1)-1).$ Then it suffices to show that $m \leq 2p$. 

Since the columns of $A$ generate the vector space $G = (\Z/p\Z)^m$, a subset of the set of the columns of $A$ gives a basis of $G$. Hence, applying an automorphism of $G$ we can assume that the canonical basis $e_1,...,e_m$ is a subset of the set of the columns of $A$. 
Consider the element $n = (1,...,1)^t \in G$, set $ n^t A = (\alpha_1,..., \alpha_r)$. Then, by Theorem \ref{torelli} (1), either 
$$1 \geq d_n = -1 + \sum_{j=1}^r \big{\langle} \frac{-\alpha_{j}}{p} \big{\rangle} \geq -1 + m  (1 - \frac{1}{p}),$$
or $$1 \geq d_{-n} = -1 + \sum_{j=1}^r \big{\langle} \frac{\alpha_{j}}{p} \big{\rangle} \geq -1 +  \frac{m}{p}.$$

Hence $m \leq 2p$. 
\end{proof}

So this concludes the proof of Theorem \ref{A}.

Clearly the above estimates are not sharp, as one can see by the following example.\\

{\bf Example}. $r=8$, $g=13$. 	
	$G = \Z/ 2 \Z \times \Z/4\Z = \langle g_1 \rangle \times \langle g_2 \rangle \subseteq  \Z/ 4 \Z \times \Z/4\Z $, $g_1 \mapsto (2,0)^t$, $g_2 \mapsto (0,1)^t$. 
	The family of covers is given by the matrix 
	
	$$A=  {\small \left( \begin{array}{cccccccccc}
			2&2&2&2&0&0&0&0\\
			0&0&0&0&1&1&1&1\\
			\end{array} \right)},$$
			so the equations can be written as:

	\begin{gather*}
	w_1^2 = \prod_{i=1}^4(x -t_i)\\
	w_2^4 = \prod_{i=5}^{8}(x-t_i).
	\end{gather*} 
	By the  Hurwitz one immediately computes $g =13$.
The non zero elements of $G$, with the identification of $G$ with its group of characters explained in section 2  are: 
$$g_1 = (1 \ 0) A, \ g_2 = (0 \ 1) A, \ g_1+g_2=(1 \ 1) A, \ g_1+2g_2= (1 \ 2)A, \  g_1+3g_2=(1 \ 3)A,  \ 2g_2=(0 \ 2)A, \ 3g_2=(0 \ 3) A. $$	

So we have: 
$$d_{g_1} = 1, \ d_{g_2} = 2, \ d_{g_1+g_2 } = 4, \ d_{g_1+2g_2} = 3, \ d_{g_1+3g_2 } = 2, \ d_{2g_2} = 1, \ d_{3g_2} =0.$$
Since  $d_{g_1+g_2}=4$, $-(g_1+g_2) = g_1 +3g_2$ and $d_{g_1 +3g_2} = 2$, we can apply Theorem \ref{torelli}(1) and conclude that this family gives a subvariety of $\A_{13}$ contained in the Torelli locus which is not totally geodesic. We could also apply Theorem \ref{torelli}(2) to conclude, since  $d_{g_1+2g_2} =3$ and $d_{g_1+2g_2}$ has order 2.

	\section{Totally geodesic subvarieties in the Prym loci}
	Denote by $ { \rgb}$ the moduli space of isomorphism classes of triples $[(C, \alpha, B)]$ where $C$ is a smooth complex projective curve of genus $g$, $B$ is a reduced effective divisor of degree $b$ on $C$ and $\alpha \in Pic(C)$ is such that $\alpha^{\otimes 2}={\mathcal O}_{C}(B)$. A point  $[(C, \alpha, B)] \in {\rgb}$ determines a double cover of $C$, $f:\tC\to C$ branched on $B$, with $\tC=\Spec({\mathcal O}_{C}\oplus \alpha^{-1})$.
	
	The Prym variety $P(C, \alpha, B)$ (also denoted by $P(\tC,C)$) associated to $[(C,\alpha, B)]$ is the connected component containing the origin of the kernel of the norm map $\Nm_{f}: J\tC \to JC$.
	When $b>0$, $\ker \Nm_{f}$ is connected. The variety  $P(C, \alpha, B)$ is a  polarised abelian variety of dimension $g-1+\frac{b}{2}$. In fact, if we denote by $\Xi$ the restriction to $P(\tC,C)$ of the principal polarisation on $J\tC$,  if $b=0$, $\Xi$ is twice a principal polarisation, hence we consider on $P(\tC, C)$ this principal polarisation. If $b>0$, we endow  $P(\tC, C)$ with the polarisation  $\Xi$ which  is of type $\delta=(1,\dots, 1, \underbrace{2,\dots,2}_{g\text{ times}})$.

If we denote by $\A^\delta_{g-1+\frac{b}{2}}$  the moduli space of abelian varieties of dimension $g-1+\frac{b}{2} $ with a polarization of type $\delta$, then we consider the Prym map 
	\[ \pgb: \rgb \lra \A^\delta_{g-1+\frac{b}{2}}, \quad [(C,\alpha,B)] \longmapsto  [(P(C,\alpha,B), \Xi)].\]

	The dual of the differential of the Prym map $\pgb$ at a generic point $[(C, \alpha, B)]$  is given by the multiplication map
	\begin{equation}
	\label{dp}
	(d\pgb)^* : S^2H^0(C, {\omega}_C \otimes \alpha) \to H^0(C, \omega^2_C\otimes\mathcal O_C(B))
	\end{equation}
	
	The multiplication map is surjective if  $\dim \rgb \leq \dim {\A}^{\delta}_{g-1+\frac{b}{2}},$ (\cite{lange-ortega}).  	Assume we have a family of Galois covers $\psi_t: \tC_t \ra \P^1 = \tC_t/\tG$, where $\tG$ is a finite group containing a central  involution $\sigma$. Then we have an exact sequence
	$ 0 \rightarrow  \langle \sigma \rangle \ra \tG \ra G \ra 0,$ and a commutative diagram  
	
	\begin{equation}
	\label{tc-c}
	\begin{tikzcd}[row sep=tiny]
	\tC_t \arrow{rr}{\phi_ t} \arrow{rd}{\psi_t} & &  \arrow{ld }{\pi_t} C_t   = \tC_t /\sx \sigma\xs  \\
	& \P^1 &
	\end{tikzcd}
	\end{equation}	
	%\label{prymetale}

	For a general element $\tC$ of the family, denote by $V:= H^0(\tC, K_{\tC}) \cong V_+ \oplus V_-$, where $V_+$ is the set if $\sigma$-invariants elements, and $V_-$ is the set if $\sigma$-anti-invariants elements. The double cover $\phi: \tC \ra C = \tC/\sx \sigma\xs$ corresponds to a triple $(C, \alpha, B)$, where $B$ is a reduced effective divisor of degree $b$ on $C$ and $\alpha$ is a line bundle on $C$ such that $\alpha^{\otimes 2} = {\mathcal O}_C(B)$. Then $V_+ \cong H^0(C, K_C)$ and $V_- \cong H^0(C, K_C \otimes \alpha) \cong H^{1,0}(P(\tC,C))$. 
	
	For the precise construction of the subvarieties of ${\A}^{\delta}_{g-1+\frac{b}{2}}$ generically contained in the Prym loci given by families of Galois covers, see  \cite[ \S 3]{fgm}. 	Given a family of Galois covers as above, consider the multiplication map $ m : S^2 V \to W= H^0(\tC, K_{\tC}^{\otimes 2})$, which is the dual of the differential of the Torelli map
	$\widetilde{j} : \M_\tg \to \A_\tg$ at the point $[\tC] \in \M_\tg$. The map $m$ is $\tG$-equivariant, hence it maps  $(S^2V)^\tG $ to $H^0(\tC, K_{\tC}^{\otimes 2})^\tG$. We have the following isomorphism:  $(S^2V)^\tG = (S^2(V_+))^{\tG} \oplus (S^2 (V_-))^{\tG}. $
	
	The dual of the differential  of the restriction of  Prym map  to the subvariety of $\rgb$ given by our family of Galois covers at the point $[(C, \alpha, B)]$  is the restriction of the multiplication map $m$ to $(S^2 (V_-))^{\tG}$.  We still denote by 
	\begin{equation}
	m : (S^2 (V_-))^{\tG} \lra H^0(\tC, K_{\tC}^{\otimes 2})^\tG
	\end{equation}
	this restriction (see \cite{fg}, \cite{fgm}).

	In \cite[Theorem~3.2]{fg}, which is a generalisation of Theorems 3.2 and 4.2 in \cite{cfgp}, it is shown that if the map $m : (S^2 (V_-))^{\tG} \lra H^0(\tC, K_{\tC}^{\otimes 2})^\tG$ is an isomorphism, then the family of Pryms yields a Shimura subvariety of $\A^\delta_{g-1+\frac{b}{2}}$. Using this criterion,  in \cite{cfgp}, \cite{fg}, \cite{fgm} many examples of Shimura subvarieties generically contained in the Prym loci have been constructed. 
	
We prove that  in the case of abelian covers, under some assumptions, the above condition is also necessary for such families of abelian covers to yield a Shimura subvariety of $\A^\delta_{g-1+\frac{b}{2}}$ generically contained in the Prym loci.

So, consider now an abelian group $\tG\subseteq (\Z/N\Z)^{m}$ containing an involution $\sigma$ as above and a $\tG$-abelian cover given by the equations \eqref{equation abelian}. 
	Let $n\in\tG$ be the element $(n_1,\dots, n_m)\in (\Z/N\Z)^{m}$ under the inclusion $\tG\subseteq (\Z/N\Z)^{m}$ with $n_i\in\Z\cap[0,N)$. 
	By \eqref{dn}, 
	$$\dim H^0(\tC,K_{\tC})_{n}=-1+\sum_{j=1}^{r}\langle-\frac{\alpha_{j}}{N}\rangle.$$
	A basis for the $\C$-vector space $H^0(\tC,K_{\tC})$ is given by the forms in \eqref{forms}: 
	$$\omega_{n,\nu}=x^{\nu} w_{1}^{n_1}\cdots w_{m}^{n_m}\prod_{j=1}^{r} (x-t_j)^{\lfloor -\frac{\tilde\alpha_j}{N}\rfloor}dx, $$ where $\tilde\alpha_j$ is as introduced above and $0\leq\nu\leq d_{n}-1=-2+\sum_{j=1}^{r}\langle-\frac{\alpha_{j}}{N}\rangle$. 
		
	The action of $\sigma$ is given by 
	$w_i\mapsto -w_i$ for some subset of $\{1,\dots, m\}$, and  $w_j\mapsto w_j$ for $j$ in the complement of this subset. 
	In fact, the elements of order 2 in $(\Z/N\Z)^{m}$ have entries either equal to zero or to $\frac{N}{2}$. Denote by $g_j= (0,...,0,\frac{N}{2},0,...,0)$ the element  of $(\Z/N\Z)^{m}$  where $\frac{N}{2}$ is in the $j$-th position. Then $\sigma = \sum_{j=1}^m \epsilon_j g_j$, where $\epsilon_j$ is either zero, or 1. Then, by the construction of abelian covers (see e.g. \cite{W}),  the action of $\sigma$ is given by: 
	\begin{equation}
	\label{sigma}
	\sigma(w_i) = (e^{\frac{2\pi i}{N}})^{ \frac{N}{2} \cdot \epsilon_i} \cdot w_i	= (-1)^{ \epsilon_i } \cdot w_i = \pm w_i.
	\end{equation}

We may then,  without loss of generality, assume that $\sigma(w_i)= -w_i$ for $i\in\{1,\dots,k\}$ for some $k\leq m$ and $\sigma(w_i)= w_i$ for the $i>k$. We recall  now \cite[Lemma 2.5]{moh}.
	\begin{LEM} \label{dimeigspace}
		The group $\tG$ acts on $H^0(\tC,K_{\tC})_{-}$ and for $n\in\tG$, $n = (n_1,...,n_m)^t$, we have $H^0(\tC,K_\tC)_{-,n}= H^0(\tC,K_\tC)_n$ if $n_1+\cdots+n_k$ is odd and $H^0(\tC,K_\tC)_{-,n}=0$ otherwise. Similar equalities hold for $H^1(\tC,\C)_{-,n}$. 
	\end{LEM}

Under these hypotheses, we have the following 
\begin{THEOR}
\label{prym}
Assume that the multiplication map $m: (S^2(V_-))^{\tG} \ra H^0(\tC, K_{\tC}^{\otimes 2})^{\tG} $ is surjective at the generic point of the family.  Then
\begin{enumerate}
\item If there exists $n \in \tG$ of order greater than 2, $\dim (V_-)_n = d_n \geq 2$ and $ \dim (V_-)_{(-n)} =d_{-n}\geq 2$, then the family of Pryms gives a subvariety of $\A^\delta_{g-1+\frac{b}{2}}$ which is not totally geodesic. 
\item If there exists $n \in \tG$ such that $n = -n$ and $\dim(V_-)_n =d_n\geq 3$, then the family of Pryms gives a subvariety of $\A^\delta_{g-1+\frac{b}{2}}$ which is not totally geodesic. 
\end{enumerate}
\end{THEOR}
	\begin{proof}
	Assume that $\tG \subseteq  (\Z/N\Z)^m$ and the family of $\tG$-covers is given by a matrix $A=(r_{ij})$ whose entries are in $\Z/N\Z$ for some $N\geq 2$. 
	Then the equations for the cover $\tG$ are 
	\begin{equation}
	w_{i}^{N}=\prod_{j=1}^{r}(x-t_{j})^{\widetilde{r}_{ij}}\quad\text{for }i=1,\dots, m,
	\end{equation}
and we can assume that the involution $\sigma$ is given by $\sigma(w_i) = -w_i$, $\forall i =1,...,k$, $\sigma(w_i) = w_i$, $\forall i \geq k+1$. Then, denoting as above by  $V = H^0(\tC, K_{\tC})$ and by 
$$V_-  = H^0(\tC, K_{\tC})^- \cong H^0(C, K_C \otimes \alpha) \cong H^{1,0}(P(\tC,C)),$$ by Lemma \ref{dimeigspace}, we have $(V_-)_n =0$, if $n_1 +...+n_k$ is even, while $(V_-)_n =V_n$ if $n_1 +...+n_k$ is odd. 

Hence if we are in case $(1)$, there exist subspaces $ \langle \omega_1, \omega_2 = x \omega_1 \rangle  \subseteq (V_-)_n =$, and $\langle \omega_3, \omega_4 = x \omega_3 \rangle \subseteq (V_-)_{(-n)}$. So the quadric $Q:= \omega_1 \odot \omega_4 - \omega_2 \odot \omega_3$ is $\tG$-invariant and it belongs to the kernel $I_2(K_C \otimes \alpha)^{\tG}$ of the multiplication map $m: (S^2(V_-))^{\tG} \ra H^0(\tC, K_{\tC}^{\otimes 2})^{\tG} $. The assumption on the surjectivity of the map $m$ is equivalent to saying that the differential of the restriction of the Prym map to the subvariety $\mathsf{X}$ of $\rgb$ given by our family of abelian covers is injective. So it is possible to study the second fundamental form of the immersion of ${\mathsf X}$ in $\A^\delta_{g-1+\frac{b}{2}}$ given by the restriction of the Prym map to ${\mathsf X}$. 

In \cite[section 2]{cf2} it is proven that if 
$$\rho_P: I_2(K_C \otimes \alpha)^{\tG} \ra (S^2 H^0(\tC, K_{\tC}^{\otimes 2}))^{\tG}$$ is the dual of the second fundamental form of the restriction of the Prym map to ${\mathsf X}$ we have 
$$\rho_P(Q)(v \odot v) = \tilde{\rho}(Q)(v \odot v),$$ $\forall v \in H^1(T_{\tC})^{\tG}$, where $\tilde{\rho}$ is the dual of the second fundamental form of the Torelli map of the family of covers $\psi: \tC \ra \tC/\tG = \PP^1$. As in the proof of Theorem \ref{torelli}, if the family of covers $\tC \ra \tC/\tG$ is contained in the hyperelliptic locus then $\tilde{\rho}(Q)$ is the section $ \tilde{Q} \cdot \hat{\eta} $, seen as an element in $S^2 H^0(K_{\tC}^{\otimes 2})$ as in \cite[Theorem 3.7]{cfg}.

Let $t \in \PP^1$ be a general point, then $\psi^{-1}(t) = \{p_1,...,p_d\}$, where $d$ is the order of $\tG$ and $p_i \neq p_j$, $\forall i \neq j$. Set $v:= \sum_{i=1}^d\xi_{p_i}$. Then clearly $v \in H^1(T_{\tC})^{\tG}$ and we have: 

\begin{gather*}
\rho_P(Q)(v \odot v) = \tilde{\rho}(Q)(v \odot v) = \sum_{i \neq j} \tilde{\rho}(Q)( \xi_{p_i} \odot \xi_{p_j}) +  \sum_{i=1}^d \tilde{\rho}(Q)( \xi_{p_i} \odot \xi_{p_i}) =\\
= -4 \pi \i \sum_{i \neq j}Q(p_i,p_j) \tilde{\eta}_{p_i}(p_j) + d \tilde{\rho}(Q)(\xi_{p_1} \odot \xi_{p_1}) =  - 2 \pi \i d \mu_2(Q)(p_1),
\end{gather*}
where $\mu_2: I_2(K_{\tC}) \ra H^0(\tC, K_{\tC}^{\otimes 4})$ is the second Gaussian map of the canonical line bundle $K_{\tC}$. 
Here we used again the $\tG$-equivariance of $\tilde{\rho}$ and the fact that $Q(p_i, p_j) = 0$, for all $i,j$, since $x(p_i) = x(p_j)$.  Hence if we take $t \in \P^1$ generic as in the proof of Theorem \ref{torelli} we have  $ \mu_2(Q)(p_1) \neq 0$. This concludes the proof of case $(1)$. 

Case $(2)$ is very similar: by assumption there exist a  subspace $\langle \omega_1, \omega_2 = x \omega_1, \omega_3 = x^2 \omega_1 \rangle \subseteq (V_-)_n = V_n$, with $n = -n$. Thus we can take the quadric $Q:= \omega_1 \odot \omega_3 - \omega_2 \odot \omega_2 \in I_2(K_C\otimes \alpha)^{\tG}.$ For a general fibre $\psi^{-1}(q) =\{p_1,...,p_d\}$, we have again 
\begin{equation}
\rho_P(Q)(v \odot v)  = - 4 \pi \i \sum_{i \neq j} Q(p_i, p_j) \cdot \eta_{p_i}(p_j) + d   \rho(Q)  (\xi_{p_1}\odot \xi_{p_1}) =-2 \pi \i d \mu_2(Q)(p_1) \neq 0,
\end{equation}
for $t$ generic.

	\end{proof}
	
	\begin{REM}
	The assumption on the surjectivity of the multiplication map $m: (S^2(V_-))^{\tG} \ra H^0(\tC, K_{\tC}^{\otimes 2})^{\tG} $ at the generic point of the family is automatically satisfied if $b \geq 6$, thanks to the Prym-Torelli theorem proved in \cite{ikeda}, \cite{no1}. 
	\end{REM}
	
	\begin{COR}
\label{Z2prym}
With the above notation,  assume that the multiplication map $m: (S^2(V_-))^{\tG} \ra H^0(\tC, K_{\tC}^{\otimes 2})^{\tG} $ is surjective, $\tG = (\Z/2\Z)^m$, and $\sigma(w_i) = -w_i$, for $i=1,...,m$. If the family of Pryms yields a totally geodesic subvariety, then $r \leq 6m$ and $\tg \leq 1 + 2^{m-1}(3m-2)$.   \end{COR}
\begin{proof}
Since $\sigma(w_i) = -w_i$, $\forall i =1,...,m$, setting as usual $e_i = (0,...,0,1,0,...,0)^t \in \tG$, by Lemma \ref{dimeigspace} we have $V_{e_i} = (V_-)_{e_i}$, for all $i =1,...,m$. Hence  by Theorem \ref{prym} (2),  if the family is totally geodesic,  we must have $d_{e_i} \leq 2$, $\forall i=1,...,m$. 
So we conclude as in Corollary \ref{Z2}. 

\end{proof}

\begin{COR}
\label{general-prym}
Assume that the multiplication map $m: (S^2(V_-))^{\tG} \ra H^0(\tC, K_{\tC}^{\otimes 2})^{\tG} $ is surjective, $\tG \subseteq (\Z/N\Z)^m$, $N \geq 3$, $\sigma(w_i) = -w_i$, for $i=1,...,m$.  Set $d:= \#\tG$. If the family of Pryms yields a totally geodesic subvariety, then $r \leq 2Nm $
and $\tg \leq 1 +d (-1 + m(N-1)) .$
 \end{COR}	

\begin{proof} 
Since $\sigma(w_i) = -w_i$, $\forall i =1,...,m$,  by Lemma \ref{dimeigspace} we have $V_{e_i} = (V_-)_{e_i}$, for all $i =1,...,m$.
Hence by Theorem \ref{prym}, for each row $e_i^t \cdot A$, $i=1,...,m$, we must have $d_{e_i} \leq 2$, if $e_i \cdot A$ has order 2, otherwise either $d_{e_i} \leq 1$, or $d_{-e_i} \leq 1$. 

As in Corollary \ref{general},  we show that in the first case we obtain $\beta_i\leq 6$, while in the second case $\beta_i \leq 2N$, $i=1,..,m$ and we conclude as in Corollary \ref{general}.

\end{proof}
	
\begin{REM}
\label{remark}
Notice that in the case $\tG= (\Z/N\Z)^m$ we can always assume that $\sigma(w_i) = -w_i$, for $i=1,...,m$, that is, $\sigma= (\frac{N}{2},...,\frac{N}{2})^t$, since, given any two elements $\sigma_1$, $\sigma_2$ of order 2, there always exists an automorphism of $\tG$ sending $\sigma_1$ to $\sigma_2$. 
\end{REM}	

\begin{proof}
Each order 2 element $\sigma_1$ of $(\Z/N\Z)^m$ has entries equal either to 0 or to $\frac{N}{2}$. Consider the element $e \in (\Z/N\Z)^m$ whose entries are either 0 or 1 and such that the zero entries are in the same positions as the ones equal to zero in $\sigma_1$ (e.g. if $\sigma_1 = (\frac{N}{2},0,...,0)^t$, then $e=(1,0,...,0)^t$).  Assume that the $i$-th entry of $\sigma_1$ is $\frac{N}{2}$. Denote as usual by $e_j$ the element having the $j$-th entry equal to one and all the other entries equal to zero.  Denote by $\phi$ the automorphism of $(\Z/N\Z)^m$ sending $e$ to $(1,1,...,1)^t$ and $e_j$ to $e_j$ for all $j \neq i$. Then clearly $\phi(\sigma_1) = (\frac{N}{2},...,\frac{N}{2})^t$. 

By \eqref{sigma} the action of $\sigma = (\frac{N}{2},...,\frac{N}{2})^t$ is $\sigma(w_i) = -w_i $, $\forall i$, since here $\epsilon_i =1$, $\forall i$.  

\end{proof}	
	
	Corollaries \ref{Z2prym}, \ref{general-prym} and Remark \ref{remark} prove Theorem \ref{B}. 
	
We give now some examples where the assumptions of Theorem \ref{prym} are satisfied.\\

		Clearly the above estimates are not sharp, as one can see by example 2. \\

		{\bf Example 1} $r=8$, $b=8$, $\tg = 5$, $g=1$. 	
	$\tG = \Z/ 2 \Z \times \Z/2\Z$. 
	The family of covers is given by the matrix 
	
	$$A=  {\small \left( \begin{array}{cccccccccc}
			1&1&1&1&0&0&0&0\\
			1&1&1&1&1&1&1&1\\
			\end{array} \right)},$$
			so the equations are: 
	
	\begin{gather*}
	w_1^2 = (x -t_1)(x-t_2) (x-t_3)(x-t_4)\\
	w_2^2 = (x-t_1)(x-t_2)(x-t_3)(x-t_4)(x-t_5)(x-t_6)(x-t_7)(x-t_8).
	\end{gather*} 
	By the Riemann Hurwitz formula, one immediately computes $\tg =5$. 
	The involution $\sigma= (1,1)^t$ acts as follows: $\sigma(w_1) = -w_1$, $\sigma(w_2) = -w_2$. The map $\tC \ra C = \tC/\langle \sigma \rangle$ ramifies over $t_i$, $i=1,...,4$, hence it has 8 ramification points,  so by the Riemann Hurwitz we see that $g = g(C) = 1$. 
	
We  compute: 
	
	 $$d_{(1,0)} = -1 + \frac{4}{2} = 1,$$ and 
	$$(V_-)_{(1,0)} =\big{ \langle} \alpha_1 = w_1  \frac{dx}{(x-t_1)(x-t_2)(x-t_3)(x-t_4)} = \frac{dx}{w_1}\big{\rangle},$$
	
	$$d_{(0,1)} = -1 + \frac{8}{2} = 3,$$ and 
	$$(V_-)_{(0,1)} =\big{ \langle} \alpha_2 = w_2 \frac{dx}{\prod_{i=1}^8 (x - t_i)} = \frac{dx}{w_2}, \alpha_3 = x \frac{dx}{w_2}, \alpha_4 = x^2 \frac{dx}{w_2} \big{\rangle}.$$
	Hence $V_- =\big{ \langle} \alpha_1, \alpha_2, \alpha_3, \alpha_4 \big{\rangle},$ and 
	$$(S^2(V_-))^{\tG} = \big{ \langle}\alpha_1 \odot \alpha_1,   \alpha_2 \odot \alpha_2, \alpha_2 \odot \alpha_3, \alpha_2 \odot \alpha_4, \alpha_3 \odot \alpha_3, \alpha_3 \odot \alpha_4, \alpha_4\odot \alpha_4  \big{\rangle} \cong \C^7.$$
	
	Since the differential of the  Prym map ${\mathsf P}_{1,8} : {\mathsf R}_{1,8} \ra \A_4^{\delta}$ is injective (\cite{no1}), the multiplication map 
	$$m: (S^2(V_-))^{\tG} \cong \C^7  \ra H^0(\tC, K_{\tC}^{\otimes 2})^{\tG} \cong \C^5 $$
	is surjective, and $\dim(V_-)_{(0,1)}= 3$, so we can apply Theorem \ref{prym} (2) and conclude that the family of Pryms is not totally geodesic. 
	
	One can also explicitly compute the kernel of the multiplication map as follows: 
We have 
$${\alpha_1}^2 = \frac{{dx}^2}{(x-t_1)(x-t_2)(x-t_3)(x-t_4)}, \ {\alpha_2}^2 = \frac{{dx}^2}{\prod_{i=1}^8 (x -t_i)}, \ \alpha_2 \alpha_3 = x \alpha_2^2,$$ 
$$\alpha_2 \alpha_4 = x^2 \alpha_2^2, \ \alpha_3^2 = x^2 \alpha_2^2, \ \alpha_3 \alpha_4 = x^3 \alpha_2^2,  \ \alpha_4^2 = x^4 \alpha_2^2.$$
We can assume that $t_5 =0$, $t_6 =1$, $t_7 =-1$, so one easily computes that 
$$a_1 \alpha_1^2 + a_2 \alpha_2^2 + a_3 \alpha_2 \alpha_3 + a_4 \alpha_2 \alpha_4 + a_5 \alpha_3^2 + a_6 \alpha_3 \alpha_4 + a_7 \alpha_4^2 =0,$$
if and only if $a_7 = -a_1$, $a_6 = t_8 a_1$, $ a_5 = -a_4 +a_1$, $a_3 = -t_8a_1$, $a_2 =0$. 
 Hence the kernel of the multiplication map 
 $$m: (S^2(V_-))^{\tG} \cong \C^7  \ra H^0(\tC, K_{\tC}^{\otimes 2})^{\tG} \cong \C^5 $$ has dimension 2 and it is generated by the quadrics 
 $$Q_1 = \alpha_2 \odot \alpha_4 - \alpha_3 \odot \alpha_3, \ Q_2 = \alpha_1 \odot \alpha_1 -t_8 \alpha_2 \odot \alpha_3 + \alpha_3 \odot \alpha_3 + t_8 \alpha_3 \odot \alpha_4 - \alpha_4 \odot \alpha_4.$$ Since $\dim (S^2(V_-))^{\tG} = 7$ and $\dim H^0(\tC, K_{\tC}^{\otimes 2})^{\tG} =8-3=5$, we conclude that $m$ is surjective. \\

		{\bf Example 2} $r=8$, $b=0$, $\tg = 33$, $g=17$. 	
	$\tG = \Z/ 4 \Z \times \Z/4\Z$. 
	The family of covers is given by the matrix 
	
	$$A=  {\small \left( \begin{array}{cccccccccc}
			1&1&1&1&0&0&0&0\\
			0&0&0&0&1&1&1&1\\
			\end{array} \right)},$$
			so the equations are: 
	
	\begin{gather*}
	w_1^4 = (x -t_1)(x-t_2) (x-t_3)(x-t_4)\\
	w_2^4 = (x-t_5)(x-t_6)(x-t_7)(x-t_8).
	\end{gather*} 
	By the  Hurwitz formula one immediately computes $\tg = 33$. 
	The involution $\sigma= (1,1)^t$ acts as follows: $\sigma(w_1) = -w_1$, $\sigma(w_2) = -w_2$. The map $\tC \ra C = \tC/\langle \sigma \rangle$ is \'etale so $g= 17$. 
	
	We  compute: 

	 $$d_{(1,0)} = -1 + 4 \cdot \frac{3}{4}  = 2 = d_{(0,1)},$$
		
	$$d_{(3,0)} = -1 + 4 \cdot \frac{1}{4} = 0 =d_{(0,3)},$$
		
	$$d_{(1,2)} = -1 + 4 \cdot \frac{3}{4} + 4 \cdot  \frac{1}{2}= 4 = d_{(2,1)},$$
	$$d_{(3,2)} = -1 + 4\cdot \frac{1}{4} + 4 \cdot \frac{1}{2}= 2 = d_{(2,3)}.$$
	
	So $$(S^2(V_-))^{\tG} =  \big{(}(V_-)_{(1,2)} \otimes  (V_-)_{(3,2)}\big{)} \oplus \big{ (}(V_-)_{(2,1)} \otimes  (V_-)_{(2,3)}\big{)}$$ and since $d_{(1,2)}  =4$ and $d_{(3,2)} =2$,  Theorem \ref{prym} (1) applies, provided that the multiplication map $m: (S^2(V_-))^{\tG} \cong \C^{16}  \ra H^0(\tC, K_{\tC}^{\otimes 2})^{\tG} \cong \C^5 $ is surjective. 
	
	We show that the restriction of $m$ to $ \big{(}(V_-)_{(1,2)} \otimes  (V_-)_{(3,2)}\big{)} $ has rank 5, hence $m$ is surjective. 
	A basis of $(V_-)_{(1,2)} $ is 
	$$\{ \beta_1 = \frac{w_1w_2^2 dx}{\prod_{i=1}^8 (x-t_i)}, \beta_2 = x\beta_1, \beta_3 = x^2\beta_1, \beta_4 = x^3\beta_1\}.$$  A basis of $(V_-)_{(3,2)} $ is $\{ \gamma_1 = \frac{w_1^3w_2^2 dx}{\prod_{i=1}^8 (x-t_i)}, \gamma_2 = x\gamma_1\}$. 
	So a basis of $ \big{(}(V_-)_{(1,2)} \otimes  (V_-)_{(3,2)}\big{)} $ is given by $\{ \beta_i \odot \gamma_j \ | \ i=1,...,4, \  j=1,2\}$. We have 
	$$\sum_{i,j}  a_{ij} \beta_i \gamma_j=0$$
	if and only if 
	$$\beta_1 \gamma_1 (a_{11} + xa_{21} + x^2a_{31} + x^3 a_{41} + xa_{12} + x^2a_{22} + x^3 a_{32} + x^4a_{42}) =0.$$
	So we get $a_{11} = a_{42} =0$, $a_{21} = -a_{12}$, $a_{31} = - a_{22}$, $a_{41} = -a_{32}$. Thus the kernel of the restriction of $m$ to $ \big{(}(V_-)_{(1,2)} \otimes  (V_-)_{(3,2)}\big{)} $ has dimension 3, hence $m$ is surjective and  the family of Pryms is not totally geodesic.

\section{Statements and Declarations}
Data sharing not applicable to this article, as no data sets were generated or analysed during the current study. \\
	
 The author states that there is no conflict of interest. 
%\newpage

\end{document}